\documentclass[12pt,twoside]{amsart}
\usepackage{amssymb,amsmath,amsthm, amscd, enumerate, mathrsfs, calligra}
\usepackage[all]{xy}
\usepackage{comment}
\usepackage{color}

\title[Zariski density of points]
{Zariski density of points with maximal arithmetic degree}
\author[K.~Sano and T.~Shibata]
{Kaoru Sano and Takahiro Shibata}
\date{}

\address{Department of Mechanical Engineering and Science, Faculty of Science and Engineering, Doshisha University, Kyoto, 610-0394, Japan.}
\email{kaosano@mail.doshisha.ac.jp}

\address{National University of Singapore, Singapore 119076, Republic of Singapore}
\email{mattash@nus.edu.sg}

\DeclareMathOperator{\Spec}{Spec}

\DeclareMathOperator{\Ker}{Ker}

\DeclareMathOperator{\Tor}{Tor}

\DeclareMathOperator{\sHom}{\mathscr{H}\kern -.3pt \mathit{om}}

\DeclareMathOperator{\AP}{AP}

\newcommand{\C}{\mathbb{C}}

\newcommand{\PP}{\mathbb{P}}

\newcommand{\Q}{\mathbb{Q}}

\newcommand{\R}{\mathbb{R}}

\newcommand{\Z}{\mathbb{Z}}

\newcommand{\A}{\mathbb{A}}

\newenvironment{parts}[0]{%
  \begin{list}{}%
    {\setlength{\itemindent}{0pt}
     \setlength{\labelwidth}{1.5\parindent}
     \setlength{\labelsep}{.5\parindent}
     \setlength{\leftmargin}{2\parindent}
     \setlength{\itemsep}{0pt}
     }%
   }%
  {\end{list}}
\newcommand{\Part}[1]{\item[\upshape#1]}

\newcommand{\setmid}{\mathrel{}\middle|\mathrel{}}

\newtheorem{thm}{Theorem}[section]
\newtheorem{lem}[thm]{Lemma}

\newtheorem{prop}[thm]{Proposition}
\newtheorem{conj}[thm]{Conjecture}
\newtheorem{ques}[thm]{Question}

\newtheorem{claim}{Claim}

\theoremstyle{definition}
\newtheorem{defn}[thm]{Definition}
\newtheorem{rem}[thm]{Remark}
\newtheorem{ex}[thm]{Example}
\newtheorem*{ack}{Acknowledgments}

\newtheorem*{notation}{Notation and Conventions}

\newtheorem{step}{Step}
\begin{document}

\begin{abstract}
Given a dominant rational self-map on a projective variety over a number field, 
we can define the arithmetic degree at a rational point. 
It is known that the arithmetic 
degree at any point is less than or equal to the first dynamical degree.
In this article, we show that there are densely many $\overline{\mathbb Q}$-rational points with maximal arithmetic degree (i.e.~whose arithmetic degree is equal to the 
first dynamical degree) for self-morphisms on projective varieties.
For unirational varieties and abelian varieties, we show that there are 
densely many rational points with maximal arithmetic degree over a sufficiently large 
number field.
We also give a generalization of a result of Kawaguchi and Silverman in the appendix.
\end{abstract}

\maketitle

\setcounter{tocdepth}{1}
\tableofcontents

\section{Introduction}\label{sec_intro}
Let $X$ be a projective variety over a number field $K$ and
$f: X \dashrightarrow X$ a dominant rational map over $K$.
Then we can define two dynamical invariants; the \textit{first dynamical degree} $\delta_f$ and
the \textit{arithmetic degree $\alpha_f(x)$ at a point $x \in X(\overline K)$}.
First dynamical degrees are geometric invariants defined by using intersection numbers,
while arithmetic degrees are arithmetic invariants defined by using height functions.
For the definitions, see Notation and Conventions below.

Relationships of those two invariants have been studied in several papers.
The following result due to Kawaguchi--Silverman and Matsuzawa is fundamental.

\begin{thm}[{\cite[Theorem 26]{KS16a}}, {\cite[Theorem 1.4]{Mat16}}]\label{thm_mat}
Let $X$ be a projective variety over a number field $K$ and
$f: X \dashrightarrow X$ a dominant rational map.
Then we have $\overline \alpha_f(x) \leq \delta_f$ for any $x \in X(\overline K)$.
\end{thm}

For the precise meanings of the symbols above, see Notation and Conventions below.
If a point $x \in X(\overline K)$ satisfies $\alpha_f(x)=\delta_f$, we say that
\textit{$x$ has maximal arithmetic degree}.
Our next concern is how rational points with maximal/non-maximal arithmetic degrees distribute.
On the distribution of points with non-maximal arithmetic degrees, \cite{MMSZ20} gave
the following conjecture.

\begin{conj}[{\cite[Conjecture 1.4]{MMSZ20}}]\label{conj_mmsz}
Let $X$ be a projective variety over a number field $K$, $f: X \to X$ a surjective morphism, and $d>0$.
Then the set
$$Z_f(d)=\{ x \in X(\overline K) \mid [K(x):K] \leq d,\ \alpha_f(x)<\delta_f \}$$
is not Zariski dense.
\end{conj}

\begin{ex}
Let $f: \PP^m_K \to \PP^m_K$ be a surjective endomorphism 
of degree $>1$ on a projective space over a number field $K$.
Using the canonical height function, it follows that 
$\alpha_f(x)<\delta_f$ if and only if $x$ is $f$-preperiodic for $x \in \PP^m(\overline K)$.
We can also see that $Z_f(d)$ is a finite set for any $d>0$,
by the Northcott property of the canonical height function (cf.~\cite{CS93}).
On the other hand, $\PP^m(K) \setminus Z_f(d)$ is Zariski dense.
\end{ex}

Conjecture \ref{conj_mmsz} are verified in many cases in \cite{MMSZ20}.
Roughly speaking, Conjecture \ref{conj_mmsz} says that the set of points of non-maximal arithmetic degrees is small.

On the other hand,
the set of points with maximal arithmetic degree should be large compared to the set of points of non-maximal arithmetic degrees.
To be precise, we study the following question.

\begin{defn}\label{defn_main}
Let $X$ be a projective variety and $f: X \dashrightarrow X$ a dominant rational map over a number field $K$
with $\delta_f>1$.
Fix an algebraic closure $\overline K$ of $K$ and let $K \subset L \subset \overline K$
be an intermediate field.
We say that
\it{$(X,f)$ has densely many $L$-rational points with maximal arithmetic degree} if
there is a subset $S \subset X_f(L)$ satisfying the following conditions.
\begin{itemize}
\item[(1)]
$S$ is Zariski dense.
\item[(2)]
$\alpha_f(x)$ exists and
$\alpha_f(x)=\delta_f$ for any $x \in S$.
\item[(3)]
$O_f(x)\cap O_f(y)=\varnothing$ for any distinct two points $x,y \in S$.
\end{itemize}
\end{defn}

\begin{ques}\label{ques_main}
Notation is as in Definition \ref{defn_main}.
Does $(X,f)$ have densely many $\overline K$-rational points with maximal arithmetic degree?
\end{ques}

\begin{rem}
If we have a point $x \in X(\overline K)$ such that $\alpha_f(x)=\delta_f$ and 
$O_f(x)$ is Zariski dense, then $S=O_f(x)$ satisfies (1) and (2) above.
But our concern here is whether we can take a Zariski dense subset of points with maximal arithmetic degree which have pairwise disjoint orbits,
so the existence of such a point does not satisfy all the claims.
\end{rem}

\begin{rem}
Assume $\delta_f=1$. Then all points have maximal arithmetic degree,
but the question is not trivial.
In this paper we treat only the case when $\delta_f>1$.
\end{rem}

Kawaguchi and Silverman \cite{KS14} first studied this problem.
They showed that the question is affirmative for special endomorphisms on the affine plane $\A^2_K$ (\cite[Theorem 3]{KS14}), although
its proof is not fully written.
For details, see the appendix.
Matsuzawa and the authors \cite{MSS18a} gave an affirmative answer
for automorphisms on smooth projective varieties (cf.~\cite[Theorem 1.7]{MSS18a}).

Our first result is an affirmative answer to Question \ref{ques_main} for the self-morphism case:

\begin{thm}\label{thmA}
Let $X$ be a projective variety and $f: X \to X$ a surjective morphism over a number field $K$
with $\delta_f>1$.
Then $(X,f)$ has densely many $\overline K$-rational points with maximal arithmetic degree.
\end{thm}

Next, we will study a stronger version of Question \ref{ques_main} for varieties which have many rational points.
A variety $X$ over a number field $K$ is called \textit{potentially dense}
if there is a finite extension field $L \supset K$ such that $X(L)$ is Zariski dense.
For potentially dense varieties, we can consider the following.

\begin{ques}\label{ques_main2}
Let $X$ be a projective variety over a number field $K$ which is potentially dense
and $f: X \dashrightarrow X$ a dominant rational map over $K$
with $\delta_f>1$.
Is there a finite extension field $L \supset K$ such that
$(X,f)$ has densely many $L$-rational points with maximal arithmetic degree?
\end{ques}

First we consider unirational varieties:

\begin{defn}
A projective variety $X$ over a field $K$ is \textit{unirational} if there is a dominant rational map
$\phi: \PP^d_K \dashrightarrow X$ from a projective space to $X$.
\end{defn}

We give an affirmative answer for unirational varieties:

\begin{thm}\label{thm_unirat}
Let $X$ be a unirational projective variety over a number field $K$ and $f: X \to X$ a surjective morphism
with $\delta_f>1$.
Then $(X,f)$ has densely many $K$-rational points with maximal arithmetic degree.
\end{thm}

Other typical examples of potentially dense varieties are abelian varieties 
(cf.~Theorem \ref{thm_Has}).
We can also give an affirmative answer for abelian varieties:

\begin{thm}\label{thm_ab}
Let $A$ be an abelian variety over a number field $K$ and $f: A \to A$ a surjective morphism
with $\delta_f>1$.
Then there exist a finite extension field $L \supset K$ such that
$(A,f)$ has densely many $L$-rational points with maximal arithmetic degree.
\end{thm}

This paper is organized as follows.
In Section \ref{sec_prelim}, we collect notation, conventions, and 
results needed later.
Theorem \ref{thmA} is proved in Section \ref{sec_gen}.
Unirational varieties and abelian varieties 
are studied in Section \ref{sec_unirat} and Section \ref{sec_ab},
respectively.
We study $\PP^1$-bundles over elliptic curves as a third type of examples 
of potentially dense varieties.
Finally, in the appendix, 
we give an affirmative answer to Question \ref{ques_main2} for special 
self-morphisms on $\A^2$, which is a generalization of \cite[Theorem 3]{KS14}.

\begin{ack}
	The first author is supported by JSPS KAKENHI Grant Number JP20K14300.
	The second author is supported by a Research Fellowship of NUS.
	The authors thank Professors Shu Kawaguchi, Joseph Silverman, 
	and De-Qi Zhang for helpful comments.
\end{ack}

\section{Preliminaries}\label{sec_prelim}

\begin{notation}
\begin{itemize}
$ \, $
\item Throughout this article, we work over a fixed number field $K$.
We fix an algebraic closure $\overline K$ of $K$.

\item A \textit{variety} means a geometrically integral separated scheme of finite type over $K$.

\item The symbols $\sim$ (resp.~$\sim_{\mathbb Q}$, $\sim_{\mathbb R}$) and
$\equiv$  mean
the linear equivalence (resp.~$\mathbb Q$-linear equivalence,
$\mathbb R$-linear equivalence) and the numerical equivalence on divisors.

\item Let $\mathbb K= \mathbb Q, \mathbb R$ or $\mathbb C$.
For a $\mathbb K$-linear endomorphism $\phi: V \to V$ on a $\mathbb K$-vector space $V$,
$\rho(\phi)$ denotes the {\it spectral radius} of $f$, that is,
the maximum of absolute values of eigenvalues (in $\C$) of $\phi$.

\item Let $X$ be a projective variety and  $f: X \dashrightarrow X$ be a dominant rational map.
We define the \textit{first dynamical degree $\delta_f$ of $f$} as
$$\delta_f=\lim_{n \to \infty} ((f^n)^*H \cdot H^{\dim X-1})^{1/n}.$$

%\item Let $f$, $g$ and $h$ be  $\mathbb R$-valued functions on a fixed domain.
%The equality $f = g + O(h)$ means that there is a positive constant $C$ such that
%$|f-g| \leq C |h|$.
%In particular,
%the equality $f=g + O(1)$ means that there is a positive constant $C$ such that
%$|f-g| \leq C$.

\item Let $X$ be a projective variety.
For an $\mathbb R$-Cartier divisor $D$ on $X$,
a function $h_D: X(\overline{K}) \to \mathbb R$ is determined up to the
difference of a bounded function.
$h_D$ is called the \textit{height function associated to $D$}.
For definition and properties of height functions, see e.g.~\cite[Part B]{HS00} or \cite[Chapter 3]{Lan83}.

\item Let $X$ be a projective variety and $f: X \dashrightarrow X$ a dominant rational map.
Fix an ample height function $h_H \geq 1$ on $X$.
\begin{itemize}
\item[(1)]
For any field $L \subset \overline K$, set
$$X_f(L)=\{ x \in X(L) \mid
f^n(x) \not\in I_f \ \mathrm{for\ every\ }n \in \Z_{\geq 0} \}.$$
\item[(2)]
For $x \in X_f(\overline K)$, we define
$$\overline \alpha_f(x)=\limsup_{n \to \infty} h_H(f^n(x))^{1/n},$$
$$\underline \alpha_f(x)=\liminf_{n \to \infty} h_H(f^n(x))^{1/n},$$
which we call the \textit{upper arithmetic degree of $f$ at $x$,
lower arithmetic degree of $f$ at $x$}, respectively.
If $\overline \alpha_f(x)=\underline \alpha_f(x)$,
then we set $\alpha_f(x)=\overline \alpha_f(x)$ and we call it the \textit{arithmetic degree of $f$ at $x$} (cf.~\cite{KS16a}).
\item[(3)]
Assume that $X$ is normal and $f$ is a morphism.
Then $\alpha_f(x)$ exists for every $x \in X(\overline K)$ (cf.~\cite{KS16b}).
\end{itemize}

\item Let $f,g: \R \to \R$ be functions and $a \in \R \cup \{ \pm \infty \}$.
Then
$$f(x) \sim g(x) \ (x \to a)$$
means that $\lim_{x \to a} \frac{f(x)}{g(x)} =1$.

\item For $a,b \in \R$,
we set $\AP(a,b)=\{a+bn \mid n \in \Z_{\geq 0} \}$.

\end{itemize}
\end{notation}

We recall some important results needed for our main theorems.
The following two results enable us to use nef canonical heights in general setting.

\begin{thm}[cf.~{\cite[Theorem 6.4 (1)]{MMSZZ20}}]\label{thm_eigen}
Let $X$ be a projective variety and $f:X \to X$ a surjective morphism.
\begin{itemize}
\item[(1)]
There exists a nef $\R$-Cartier divisor $D$ on $X$ such that $D \not\equiv 0$ and $f^*D \equiv \delta_f D$.
\item[(2)]
Assume that $X$ is normal.
Let $D$ be an $\R$-Cartier divisor on $X$ such that $f^*D \equiv \delta_f D$.
Then there exists an $\R$-Cartier divisor $D'$ such that $D' \equiv D$ and $f^*D' \sim_\R \delta_fD'$.
\end{itemize}
\end{thm}

\begin{thm}[cf.~{\cite[Theorem 1.1]{CS93}}]\label{thm_canht}
Let $X$ be a projective variety, $f: X \to X$ a surjective morphism, and  $D$ an $\R$-Cartier divisor on $X$ such that $f^*D \sim_{\R} \lambda D$ with $\lambda >1$.
Take a height function $h_D$ associated to $D$.
\begin{itemize}
\item[(1)]
The limit
$$\widehat h_{D,f}(x)=\lim_{n \to \infty} \frac{h_D(f^n(x))}{\lambda^n}$$
converges for any $x \in X(\overline K)$.
\item[(2)]
$\widehat h_{D,f}=h_D+O(1)$.
\item[(3)]
$\widehat h_{D,f} \circ f = \lambda \widehat h_{D,f}$.
\item[(4)]
If $\widehat h_{D,f}(x)>0$, then $\alpha_f(x) \geq \lambda$.
\end{itemize}
\end{thm}

%The following lemma enables us to reduce an endomorphism to its positive power.
%
%\begin{lem}\label{lem_power}
%Let $X$ be a projective variety and $f: X \to X$ a dominant morphism.
%Take an intermediate field $K \subset L \subset \overline K$ and $m \in \Z_{>0}$.
%If $(X,f^m)$ has densely many $L$-rational points with maximal arithmetic degree,
%then $(X,f)$ has densely many $L$-rational points with maximal arithmetic degree.
%\end{lem}
%
%\begin{proof}
%Let $S \subset X(L)$ be a Zariski dense subset such that 
%\begin{itemize}
%\item $\alpha_{f^m}(x)=\delta_{f^m}$ for $x \in S$.
%\item $O_{f^m}(x) \cap O_{f^m}(y)=\varnothing$ if $x,y \in S$ and $x \neq y$.
%\end{itemize}
%Then $\alpha_f(x)=\delta_f$ for any $x \in S$.
%Let $\{Y_n\}_{n=1}^\infty$ be the set of proper closed subsets of $X_{\overline K}$.
%Assume that we have $x_1, \ldots,x_k \in S$ such that 
%\begin{itemize}
%\item $x_i \not\in Y_i$ for $1 \leq i \leq k$.
%\item $O_f(x_i) \cap O_f(x_j)=\varnothing$ for $1 \leq i, j \leq k$ with $i \neq j$.
%\end{itemize}
%\end{proof}

\begin{lem}\label{lem_cover}
Let $X, Y$ be projective varieties and $f: X \to X$, $g: Y \to Y$ surjective morphisms.
Let $\pi:X \to Y$ be a finite surjective morphism such that $g \circ \pi=\pi \circ f$.
Take an intermediate field $K \subset L \subset \overline K$.
If $(X,f)$ has densely many $L$-rational points with maximal arithmetic degree,
then $(Y,g)$ has densely many $L$-rational points with maximal arithmetic degree.
\end{lem}

\begin{proof}
Let $S \subset X(L)$ be a Zariski dense subset such that 
\begin{itemize}
\item $\alpha_{f}(x)=\delta_{f}$ for $x \in S$.
\item $O_{f}(x) \cap O_{f}(x')=\varnothing$ for $x,x' \in S$ with $x \neq x'$.
\end{itemize}
Let $\{Y_n\}_{n=1}^\infty$ be the set of proper closed subsets of $Y_{\overline K}$.
Assume that we have $x_1, \ldots,x_k \in S$ such that 
\begin{itemize}
\item $\pi(x_i) \not\in Y_i$ for $1 \leq i \leq k$.
\item $O_g(\pi(x_i)) \cap O_g(\pi(x_j))=\varnothing$ for $1 \leq i, j \leq k$ with $i \neq j$.
\end{itemize}
Now there are infinitely many points in $S$ which are not contained in $\pi^{-1}(Y_{k+1})$.

Suppose that the $g$-orbit of any $y \in \pi(S) \setminus Y_{k+1}$ intersects 
$O_g(\pi(x_i))$ for some $1 \leq i \leq k$.
Then we have a infinite subset $S' \subset S \setminus \pi^{-1}(Y_{k+1})$ 
and some $1 \leq i \leq k$ such that 
$O_g(\pi(x)) \cap O_g(\pi(x_i)) \neq \varnothing$ for any $x \in S'$.
This implies that $O_f(x) \cap O_f(x') \neq \varnothing$ for some 
$x, x' \in S'$ with $x \neq x'$, a contradiction.
So we have $x_{k+1} \in S \setminus Y_{k+1}$ such that 
$O_g(\pi(x_{k+1})) \cap O_g(\pi(x_i))=\varnothing$ for $1 \leq i \leq k$.
Continuing this process, the assertion follows.
\end{proof}

We will use the following estimate of the number of rational points of bounded heights 
on a projective space.

\begin{thm}[cf.~{\cite{Sch64}}, {\cite[Theorem B.6.2]{HS00}}]\label{thm_rate}
Let $H$ be the absolute multiplicative natural height function on $\PP^N_K$.
Then there exists a constant $C=C(K,N)>0$ such that
$$|\{x \in \PP^N(K) \mid H(x) <T \}| \sim C T^{N+1} \ (T \to \infty).$$
\end{thm}

\begin{rem}
The statements in \cite{Sch64} and \cite{HS00} give more precise estimates of
the growth rate, but the above form is sufficient for our purpose.
\end{rem}

The following result guarantees the potential density of abelian varieties.

\begin{thm}[cf.~{\cite[Proposition 4.2]{Has03}}]\label{thm_Has}
Let $A$ be an abelian variety.
Then there exists  $x \in A(\overline K)$ such that
$\Z_{> 0}x =\{nx \mid n \in \Z_{> 0} \}$ is Zariski dense.
\end{thm}

Kawaguchi and Silverman gave the following precise description of the zero locus of
nef canonical height functions on abelian varieties, which we will use in Section \ref{sec_ab}.

\begin{thm}[cf.~{\cite[Theorem 29]{KS16b}}]\label{thm_ks}
Let $A$ be an abelian variety, $f:A \to A$ an isogeny with $\delta_f>1$,
and $D \not\equiv 0$ a symmetric nef $\R$-Cartier divisor on $A$ such that $f^*D \equiv \delta_f D$.
Then there is a proper abelian subvariety $B \subset A_{\overline K}$ such that
$$\{ x \in A(\overline K) \mid \widehat h_{D,f}(x)=0 \} = B(\overline K)+\Tor(A(\overline K)).$$
\end{thm}

The following decomposition of a self-isogeny on an abelian variety is also needed in Section \ref{sec_ab}.

\begin{lem}[cf.~{\cite[Proof of Theorem 2]{Sil17}}]\label{lem_sil}
Let $A$ be an abelian variety and  $\phi: A \to A$ an isogeny.
Then we can take two abelian subvarieties $A_1,A_2$ of $A$ such that we have:
\begin{itemize}
\item
The addition map $\mu: A_1 \times A_2 \to A$ is an isogeny.
\item
$A_i$ is $\phi$-invariant for $i=1,2$.
We set $\phi_i=\phi|_{A_i}$.
\item
$1_{A_1} -\phi_{1}: A_1 \to A_1$ is surjective.
\item
$\delta_{\phi_2}=1$.
\end{itemize}
\end{lem}

\section{General case}\label{sec_gen}
In this section, we show that there are densely many $\overline K$-rational points
of maximal arithmetic degree for arbitrary endomorphisms on projective varieties.

We use the following lemma several times in the rest of this paper.

\begin{lem}\label{lem_key}
Let $X$ be a quasi-projective variety and $f: X \to X$ a dominant morphism.
Take an intermediate field $K \subset L \subset \overline K$.
Let (P) and (Q) be  conditions on a point of $X(L)$.
Assume that (P) implies (Q).
Suppose the following.

\vspace{0.15in}
\noindent \rm ($*$): \it Given finitely many points $x_1,\ldots,x_k \in X(L)$ satisfying (Q)
and a proper closed
subset $Y \subset X_{\overline K}$,
there exists a point $x \in X(L)$ such that
$x$ satisfies (P), $x \not\in Y$, and
$O_f(x) \cap O_f(x_i)=\varnothing$ for $1 \leq i \leq k$.
\vspace{0.15in}

Then there is a subset $S \subset X(L)$ such that
\begin{itemize}
\item[(1)]
$S$ is Zariski dense.
\item[(2)]
Every point of $S$ satisfies (P).
\item[(3)]
$O_f(x)\cap O_f(y)=\varnothing$ for any distinct two points $x,y \in S$.
\end{itemize}
\end{lem}

\begin{proof}
Let $\{Y_i\}_{i=1}^\infty$ be the set of all proper closed subsets of $X_{\overline K}$.
By assumption, there is $x_1 \in X(L)$ satisfying (P) and $x_1 \not\in Y_1$.

Assume that we have $x_1, \ldots, x_k \in X(L)$ satisfying
\begin{itemize}
\item[(i)]
$x_i \not\in Y_i$ for $1 \leq i \leq k$.
\item[(ii)]
$x_i$ satisfies (P) for $1 \leq i \leq k$.
\item[(iii)]
$O_f(x_i)\cap O_f(x_j)=\varnothing$ for $1 \leq i,j \leq k$ with $i \neq j$.
\end{itemize}
Then ($*$) implies that there is $x_{k+1} \in X(L)$ satisfying (P) such that
$x_{k+1} \not\in Y_{k+1}$
and $O_f(x_{k+1}) \cap O_f(x_i)=\varnothing$ for $1 \leq i \leq k$.
Continuing this process, we obtain $S=\{x_i\}_{i=1}^\infty$ satisfying all the claims.
\end{proof}

This formulation enables us to replace a given self-morphism to its iteration.

\begin{lem}\label{lem_it}
Notation is as in Lemma \ref{lem_key}.
Further we assume that (Q) is preserved under the action of $f$.
If $(X,f^m)$ satisfies ($*$) for some $m>0$,
then $(X,f)$ also satisfies ($*$).
\end{lem}

\begin{proof}
Take finitely many points $x_1,\ldots,x_k \in X(L)$ satisfying (Q)
and a proper closed
subset $Y \subset X_{\overline K}$.
Then $f^j(x_i)$ satisfies (Q) for $0 \leq j \leq m-1$, $1 \leq i \leq k$ since
(Q) is preserved under the action of $f$.
By assumption, there is $x \in X(L)$ satisfying (P) such that
$x \not\in Y$ and
$O_{f^m}(x) \cap O_{f^m}(f^j(x_i))=\varnothing$ for $1 \leq i \leq k$, $0 \leq j \leq m-1$.
Then it is clear that $O_f(x) \cap O_f(x_i)=\varnothing$ for $1 \leq i \leq k$.
\end{proof}

%\begin{lem}\label{lem_cov}
%Notation is as in Lemma \ref{lem_key}.
%Let $g: Y \to Y$ be another dominant self-morphism on a quasi-projective variety
%and $\pi: X \to Y$  a finite surjective morphism.
%Then $(X,f)$ satisfies ($*$) if and only if $(Y,g)$ satisfies ($*$).
%\end{lem}
%
%\begin{proof}
%Assume that $(X,f)$ satisfies ($*$).
%Take $y_1,\ldots,y_k \in Y(L)$ satisfying (P) and a proper closed subset $W \in Y_{\overline K}$.
%Take any $x_i \in \pi^{-1}(y_i)(
%\end{proof}

\begin{lem}\label{lem_dense}
Let $X$ be a projective variety and
$D \not\equiv 0$ a nef $\R$-Cartier divisor on $X$.
Take $A>0$ and $B \in \R$, and let
$T(D,A,B)$ be the set
of the points $x \in X(\overline K)$ such that $[K(x):K]$ has a prime number $p>A$ as a factor and $h_D(x)>B$.
Then $T(D,A,B)$ is Zariski dense.
\end{lem}

\begin{proof}
Take any proper closed subset $Y$ of $X_{\overline K}$.
Take very ample divisors $H_1, \ldots,H_{\dim X-1}$ such that
$C=H_1 \cap \cdots \cap H_{\dim X-1} \not\subset Y$.
Restricting $D$ to $C$, we may assume that $X$ is a projective curve and $D$ is ample,
and it is sufficient to show that $T(D,A,B)$ is an infinite set.
Now we can write $D=\sum_i a_i D_i$ with $a_i>0$ and $D_i$ being ample Cartier divisors,
so we may assume that $D$ is an ample Cartier divisor.

Let $\nu: X^\nu \to X$ be the normalization of $X$.
We can take $e>0$ such that $[K(x'):K(\nu(x'))] \leq e$ for any $x' \in X^\nu(\overline K)$.
Choose a height function $h_{\nu^*D}$ associated to $\nu^*D$ as $h_{\nu^*D}=h_D \circ \nu$.
Take any $x' \in T(\nu^*D,eA,B)$.
Then
$$[K(\nu(x')):K]=\frac{[K(x'):K]}{[K(x'):K(\nu(x'))]}$$
still has a prime factor $p>eA$ and $h_D(\nu(x'))=h_{\nu^*D}(x')>B$,
so $\nu(x') \in T(D,A,B)$.
Therefore we have $\nu(T(\nu^*D,eA,B)) \subset T(D,A,B)$.
So we may assume that $X$ is a smooth projective curve.

Take a finite surjective morphism $f: X \to \mathbb P^1$
with $\pi^*\mathcal O(1) \sim mD$ for some $m>0$.
Take $e'>0$ such that $[K(x):K(f(x))] \leq e'$ for any $x \in X(\overline K)$.
Now we have $m h_D=h \circ f+O(1)$, where $h=h_{\mathcal O(1)}$ is the naive height function on $\PP^1$.
So there is $M>0$ such that $|mh_D - h \circ f| \leq M$.
As in the previous paragraph, we have
$f^{-1}(T(\mathcal O(1),A,mB+M)) \subset T(D,A,B)$.
Therefore we may assume that $X=\mathbb P^1$, $D=\mathcal O(1)$,
and $h_D=h$.

Take $\alpha \in \overline K$ such that $[\Q(\alpha):\Q]=p$ where
$p$ is a prime number with $p > A$ and $p>m=[K:\Q]$.
Then $p \mid [K(\alpha):K]$ since $m<p$.
But now $[K(\alpha):K] \leq [\Q(\alpha):\Q]=p$, so $[K(\alpha):K]=p$.
There are infinitely many points in $\A^1(K(\alpha)) \setminus \A^1(K)=K(\alpha) \setminus K$ and
$[K(x):K]=p$ for any $x \in K(\alpha) \setminus K$.
So there are infinitely many $x \in \A^1(K(\alpha))$ such that
$[K(x):K]=p>A$ and $h(x)>B$ by the Northcott finiteness property.
\end{proof}

\begin{proof}[Proof of Theorem \ref{thmA}]
Taking the normalization, we may assume that $X$ is normal.
Let $D$ be a nef $\R$-Cartier divisor on $X$ such that $D \not\equiv 0$ and
$f^*D \sim_{\R} \delta_f D$.
Then we have a canonical height function $\widehat h_{D,f}$ (cf.~Theorem  \ref{thm_canht}).
According to Lemma \ref{lem_key},
it is sufficient to show the following claim,
where we set 
\begin{itemize}
\item $x \in X(\overline K)$ satisfies (P) if $\widehat h_{D,f}(x)>0$.
\item (Q) is unconditional.
\end{itemize}
Note that $\widehat h_{D,f}(x)>0$ implies $\alpha_f(x)=\delta_f$ (cf.~Theorem \ref{thm_canht} (4)).

\begin{claim}
Given finitely many points $x_1,\ldots,x_k \in X(\overline K)$
%satisfying $\widehat h_{D,f}(x_i)>0$
%$(1 \leq i \leq k)$
and a proper closed
subset $Y \subset X_{\overline K}$,
there exists a point $x \in X(\overline K)$ such that
$\widehat h_{D,f}(x)>0$, $x \not\in Y$, and
$O_f(x) \cap O_f(x_i)=\varnothing$ for $1 \leq i \leq k$.
\end{claim}

Take $A>0$ such that $[K(x):K(f(x))] \leq A$ for any $x \in X(\overline K)$ and $[K(x_i):K] \leq A$ for $1 \leq i \leq k$.
Then the set $T(D,A,0)$ is dense by Lemma \ref{lem_dense}.
So there is a point $x \in X(\overline K)$ such that
$[K(x):K]$ has a prime factor $p>A$, $\widehat h_{D,f}(x)>0$, and $x \not\in Y$.

Suppose $O_f(x) \cap O_f(x_l) \neq \varnothing$ for some $1 \leq l \leq k$.
Then there exist $N,M \in \Z_{>0}$ satisfying $f^N(x)=f^M(x_l)$.
Now we have
\begin{align*}
[K(x):K]
&=\left( \prod_{i=0}^{N-1} [K(f^i(x)):K(f^{i+1}(x))] \right) \cdot [K(f^N(x):K] \\
&=\left( \prod_{i=0}^{N-1} [K(f^i(x)):K(f^{i+1}(x))] \right) \cdot [K(f^M(x_l):K].
\end{align*}
Here $[K(f^i(x)):K(f^{i+1}(x))] \leq A$ for $0 \leq i \leq N-1$ and
$[K(f^M(x_l):K] \leq [K(x_l):K] \leq A$.
So $[K(x):K]$ does not have $p$ as a factor, which is a contradiction.
So $O_f(x) \cap O_f(x_l) = \varnothing$ for $1 \leq l \leq k$.
\end{proof}

\section{Unirational varieties}\label{sec_unirat}
In this section, we prove Theorem \ref{thm_unirat}.
In fact, we prove a bit more general statement here:

\begin{thm}\label{thm_unirat2}
Let $X$ be a projective variety and $f: X \to X$ a surjective morphism with $\delta_f>1$.
Assume the following condition:

\vspace{0.15in}
\noindent \rm (\dag): \it There is a nef $\R$-Cartier divisor $D$ on $X$ such that
$f^*D \sim_\R \delta_fD$ and
for any proper closed subset $Y \subset X_{\overline K}$,
there exists a morphism $g: \PP^1_K \to X$ such that $g(\PP^1_K) \not\subset Y$
and $g^*D$ is ample.
\vspace{0.15in}

\noindent Then $(X,f)$ has densely many $K$-rational points with maximal arithmetic degree.
\end{thm}

Theorem \ref{thm_unirat2} implies Theorem \ref{thm_unirat} due to the following lemma.

\begin{lem}\label{lem_p1}
Let $X$ be a unirational projective variety and $f: X \to X$ a surjective morphism with
$\delta_f>1$.
Then $(X,f)$ satisfies \rm (\dag) \it in Theorem \ref{thm_unirat2}.
\end{lem}

\begin{proof}
Taking normalization, we may assume that $X$ is normal.
Take a nef $\R$-Cartier divisor $D \not\equiv 0$ on $X$ such that $f^*D \sim_\R \delta_f D$.
Let $\phi: \PP^d_K \dashrightarrow X$ be a dominant rational map from a projective space.
Take a resolution of the indeterminacy of $\phi$:
\[
\xymatrix{
\tilde X \ar[rd]^{\beta} \ar[d]_{\alpha} \\
\PP^d_K \ar@{.>}[r]^\phi & X
}
\]
Then $E=\alpha^* \phi^*D -\beta^*D=\alpha^* \alpha_* \beta^*D -\beta^*D$ is effective by the negativity lemma (cf.~\cite[Lemma 3.39]{KM98}).
So $\alpha^*\phi^*D=\beta^*D+E$ is pseudo-effective and numerically non-trivial.
Hence $\phi^*D$ is pseudo-effective and numerically non-trivial, so ample.

As the intersection of general hyperplanes,
we can take a line $Z \subset \PP^d_{K}$
such that $Z \cap I_\phi=\varnothing$ and $\phi(Z) \not\subset Y$.
Set $g:\PP^1_K =Z \overset{\iota}{\hookrightarrow} \PP^d_{K} \overset{\phi}{\dashrightarrow} X$.
Then $g^*D=(\phi \circ \iota)^*D=\iota^*\phi^*D$ (cf.~\cite[Lemma 4.7]{MSS18b}),
so it is ample.
\end{proof}

\begin{proof}[Proof of Theorem {\ref{thm_unirat2}}]
Take a canonical height function $\widehat h_{D,f}$ associated to $D$.
According to Lemma \ref{lem_key},
it is sufficient to show the following claim, where we set 
\begin{itemize}
\item $x \in X(K)$ satisfies (P) if $\widehat h_{D,f}(x)>0$.
\item (Q) is same as (P).
\end{itemize}

\begin{claim}
Given finitely many points $x_1,\ldots,x_k \in X(K)$
satisfying $\widehat h_{D,f}(x_i)>0$
$(1 \leq i \leq k)$
and a proper closed
subset $Y \subset X_{\overline K}$,
there exists a point $x \in X(K)$ such that
$\widehat h_{D,f}(x)>0$, $x \not\in Y$, and
$O_f(x) \cap O_f(x_i)=\varnothing$ for $1 \leq i \leq k$.
\end{claim}

\noindent We can take a morphism $g:\PP^1_K \to X$ over $K$ such that
$g(\PP^1) \not\subset Y$ and $g^*D$ is ample by Lemma \ref{lem_p1}.
Write $g^*D \sim \mathcal O(r)$.
Then it follows from the functoriality of height functions that there exists $c>0$ such that
$| \widehat h_{D,f}(g(a)) - r h(a) | <c$ for any $a \in \PP^1(\overline K)$,
where $h$ is the natural height function on $\PP^1$.
Set
$$V=\{ \widehat h_{D,f}(f^m(x_i)) \mid m \geq 0,\ 1 \leq i \leq k \}
\cup \{ \widehat h_{D,f}(x) \mid x \in g(\PP^1) \cap Y \}$$
and
$v= \{ \log h \mid h \in V \} \subset \R \cup \{ -\infty \}$.
Write $V=\{ h_1, h_2, \ldots \}$ as satisfying $h_1<h_2<\cdots$.
Now we have
$$v=\bigcup_{i=1}^k \AP(\log \widehat h_{D,f}(x_i), \log \delta_f) \cup (\mathrm{finite\ set}).$$
So there exists $M>0$ such that
$$v \cap [M+n\log \delta_f,M+(n+1) \log \delta_f]
=v \cap [M,M+\log \delta_f]+n \log \delta_f$$
for any $n \in \Z_{\geq 0}$.
Let $d'>1$ be the number such that
$$\log d'=\min \{ |\alpha-\beta|
\mid  \alpha, \beta \in v \cap [M,M+2\log \delta_f],\ \alpha \neq \beta \}.$$
Take $d \in \R$ satisfying $1<d<d'$.

Now we have
$$\liminf_{l \to \infty} \frac{h_{l+1}-c}{h_l+c}=d'>d.$$
So we have
$\log h_l>M$ and $(h_{l+1}-c)/(h_l+c)>d$ for a sufficiently large $l \in \Z_{>0}$.

Set
$$N_K(p,q)=\{ a \in \PP^1(K) \mid
p \leq h(a)<q \}.$$
Using Theorem \ref{thm_rate}, there is a constant $C>0$ such that
$$N_K \left(\frac{h_l+c}{r},\frac{d(h_l+c)}{r}\right) \sim C \exp \left( \frac{2d}{r} (h_l+c) \right) \to \infty \ (l \to \infty).$$
So, taking $l$ sufficiently largely, we can take $a \in \PP^1(K)$ such that
$h_l+c \leq rh(a) < d(h_l+c)$.
So we have $h_l+c \leq rh(a) < h_{l+1}-c$.

Set $x=g(a)$.
Then we have
\begin{align*}
\widehat h_{D,f}(x)-h_l
&=(\widehat h_{D,f}(x)-rh(a))+(rh(a)-h_l)\\
&> -c+c=0
\end{align*}
and
\begin{align*}
h_{l+1}-\widehat h_{D,f}(x)
&=(h_{l+1}-rh(a))+(rh(a)-\widehat h_{D,f}(x) \\
&> c-c=0.
\end{align*}
So $h_l < \widehat h_{D,f}(x) < h_{l+1}$.
Hence $\widehat h_{D,f}(f^m(x)) \not\in V$ for any $m \geq 0$
since $v \cap [M,\infty)$ is $\log \delta_f$-periodic.
Thus $O_f(x) \cap O_f(x_i)=\varnothing$ for $1 \leq i \leq k$.
\end{proof}

\section{Abelian varieties}\label{sec_ab}
In this section, we prove Theorem \ref{thm_ab}.

The following is a censequence of Faltings's theorem.

\begin{lem}\label{lem_int}
Let $A$ be an abelian variety,
$x \in A(\overline K)$ a point such that $\Z_{> 0}x$ is dense,
and $Y \subset A_{\overline K}$ a proper closed subset.
Then $Y \cap \Z_{> 0} x$ is a finite set.
\end{lem}

%\begin{proof}
%Let $\tau_{x}:A \to A$ be the translation map by $x$.
%Then $\Z_{> 0}x=O_{\tau_{x}}(x)$.
%Applying the dynamical Mordell--Lang theorem for \'etale endomorphisms to $\tau_x$
%(cf.~\cite{BGT10}),
%the claim follows.
%\end{proof}

\begin{lem}\label{lem_num}
Let $\delta \in \R$ be an algebraic number and $l_1,\ldots,l_k \in \Z_{>0}$.
Then there exist infinitely many $l \in \Z_{>0}$ such that
$$\left( \frac{l_i}{l} \right)^2 \not\in \delta^\Z =\{ \delta^n \mid n \in \Z \}.$$
\end{lem}

\begin{proof}
Let $R$ be the ring of integers of $\Q(\delta)$.
Then the set
$$\mathcal P=\{(0)\} \cup \{ \frak p \in \Spec R \setminus \{ (0) \}
\mid v_{\frak p}(l_i) \neq 0 \ \mathrm{for\ some\ } i \ \mathrm{or}\ v_{\frak p}(\delta) \neq 0 \}$$
is finite.
Take $\frak p \in \Spec R \setminus \mathcal P$.
Now $\frak p \cap \Z=(p)$ for some prime number $p$.
Let $l$ be any member of $(p)$.
Then we have
$$v_{\frak p} \left( \left( \frac{l_i}{l} \right)^2 \right) \neq 0=v_{\frak p}(\delta^n).$$
for any $1 \leq i \leq k$ and $n \in \Z$.
So the assertion follows.
\end{proof}

\begin{proof}[Proof of Theorem {\ref{thm_ab}}]
Set $\delta=\delta_f$ and
$f=\tau_a \circ \phi$ where $\phi:A \to A$ is an isogeny and $\tau_a:A \to A$ is the translation by $a \in A(K)$.
Applying Lemma \ref{lem_sil}, there are abelian subvarieties $A_1, A_2 \subset A$ such that
\begin{itemize}
\item
The addition map $m: A_1 \times A_2 \to A$ is an isogeny,
\item
$A_i$ is $\phi$-invariant for $i=1,2$
(we set $\phi_i=\phi|_{A_i}$),
\item
$1_{A_1} -\phi_{1}: A_1 \to A_1$ is surjective, and
\item
$\delta_{\phi_2}=1$.
\end{itemize}
Write $a=a_1+a_2$ where $a_i \in A_i(\overline K)$ ($i=1,2$)
and take $b \in A_1(\overline K)$ such that $b-\phi_1(b)=a_1$.
Now we have the following commutative diagram:
\[
\xymatrix{
A_1 \times A_2 \ar[r]^{\phi_1 \times f_2} \ar[d]_{\tau_{b} \times 1_{A_2}} & A_1 \times A_2 \ar[d]^{\tau_{b} \times 1_{A_2}}\\
A_1 \times A_2 \ar[r]^{f_1 \times f_2} \ar[d]_{\mu} & A_1 \times A_2 \ar[d]^{\mu}\\
A \ar[r]^{f}& A
}
\]
Here $\mu: A_1 \times A_2 \to A$ is the addition map.
Note that $\delta=\delta_{f_1}=\delta_{\phi_1}$.

Take a point $x \in A_1(\overline K)$ such that
$\Z_{>0}x$ is Zariski dense in $(A_1)_{\overline K}$ (cf.~Theorem \ref{thm_Has}).
Extending $K$ if necessary, we may assume that all concerned are defined over $K$.
Let $z=lx+b+y$ where $l \in \Z_{>0}$ and $y \in A_2(K)$.
Now we have $\alpha_{\phi_1}(lx)=\delta$ since
$\widehat h_{D,\phi_1}(lx)=l^2 \widehat h_{D,\phi_1}(x)$ and $\widehat h_{D,\phi_1}(x)>0$ thanks to
Theorem \ref{thm_ks}.
We compute
\begin{align*}
\alpha_f(z)
&=\alpha_{\phi_1 \times f_2}(lx,y) \\
&=\max \{ \alpha_{\phi_1}(lx), \alpha_{f_2}(y) \} \\
&= \max \{\delta,1 \} = \delta.
\end{align*}
So, according to Lemma \ref{lem_key}, it is sufficient to prove the following claim,
where we set 

\begin{itemize}
\item $z \in A(K)$ satisfies (P) if $z=lx+b+y$ for some $l \in \Z_{>0}$ and $y \in A_2(K)$.
\item (Q) is same as (P).
\end{itemize}

\begin{claim}
Given finitely many points $z_1,\ldots,z_k \in A(K)$
such that $z_i=l_ix+b+y_i$ $(l_i \in \Z_{>0},\ y_i \in A_2(K))$
for $1 \leq i \leq k$
and a proper closed
subset $Y \subset A_{\overline K}$,
there exists a point $z \in A(K)$ such that
$z=lx+b+y$ $(l \in \Z_{>0},\ y \in A_2(K))$, $z \not\in Y$, and
$O_f(z) \cap O_f(z_i)=\varnothing$ for $1 \leq i \leq k$.
\end{claim}

Set $\tilde Y=\mu^{-1}(Y)$ and $V=\pi_1(\tilde Y)$, where
$\pi_1: A_1 \times A_2 \to A_1$ is the first projection.
By Lemma \ref{lem_int} and Lemma \ref{lem_num}, we can take $l \in \Z_{>0}$ such that
$(l/l_i)^2 \not\in \delta^\Z$ for $1 \leq i \leq k$ and $lx \not\in V$ if $V \neq A_1$.
Extending $K$ to make $A_2(K)$ Zariski dense,
we can take $y \in A_2(K)$ as satisfying
\begin{itemize}
\item $y$ is arbitrary if $V \neq A_1$, and
\item $(lx+b,y) \not\in \tilde Y$ if $V = A_1$.
\end{itemize}
Then $z=lx+b+y \not\in Y$.

Finally suppose $f^m(z_i)=f^n(z)$ for some $m,n \geq 0$ and $1 \leq i \leq k$.
This means that
$$(\phi_1^m(l_ix)+b,f_2^m(y_i))-(\phi_1^n(lx)+b,f_2^n(y))
\in \Ker \mu.$$
Now $\Ker \mu$ is a finite subgroup of $A_1 \times A_2$, so
$$N\left( (\phi_1^m(l_ix)+b,f_2^m(y_i))-(\phi_1^n(lx)+b,f_2^n(y)) \right)=0$$
for a positive integer $N$.
In particular, $N\phi_1^m(l_ix)=N\phi_1^n(lx)$.
Substituting this equality to $\widehat h_{D,\phi_1}$, we have
$$N^2\delta^m l_i^2 \widehat h_{D,\phi_1}(x)=N^2 \delta^n l^2 \widehat h_{D,\phi_1}(x).$$
Thus $(l/l_i)^2=\delta^{m-n}$, a contradiction.
So the claim follows.
\end{proof}

\section{$\PP^1$-bundles over elliptic curves}\label{sec_bundle}
By the potential density of abelian varieties, projective bundles over abelian varieties are also potentially dense.
We show the density of rational points with maximal arithmetic degree for $\PP^1$-bundle over elliptic curves.
We refer to Section $5$ and $6$ of \cite{MSS18a} for the properties of self-morphisms of $\PP^1$-bundles over elliptic curves.

\begin{thm}\label{thm_bundle}
Let $\pi : X \to C$ be a $\PP^1$-bundle
over an elliptic curve $C$ and $f: X \to X$ a surjective morphism with $\delta_f>1$.
Then there is a finite extension $L \supset K$ such that
$(X,f)$ has densely many $L$-rational points with maximal arithmetic degree.
\end{thm}

\begin{proof}
Every fiber of $\pi$ is mapped onto a fiber by $f$.
So $f$ induces a self-morphism $g: C \to C$
satisfying $\pi \circ f = g \circ \pi$.

\vspace{0.1in}

\underline{Case 1}: $\delta_f= \delta_g \ (=\delta)$.

\noindent Theorem \ref{thm_ab} implies that there exist a finite extension $L \supset K$
and a subset $T \subset C(L)$ such that
\begin{itemize}
\item[(1)]
$T$ is Zariski dense.
\item[(2)]
$\alpha_g(y)=\delta$ for any $y \in T$.
\item[(3)]
$O_f(y)\cap O_f(y')=\varnothing$ for any distinct two points $y,y' \in T$.
\end{itemize}
Let $\{Y_n\}_{n=1}^\infty$ be the set of all proper closed subsets of $X_{\overline K}$.
For each $n \in \Z_{\geq 1}$,
take $x_n \in X(L)$ inductively as follows:
if $\pi(Y_n) = C$, take any $y \in T \setminus \{\pi(x_1),\ldots,\pi(x_{n-1})\}$
and take $x_n \in X_y(L)$ satisfying $x_n \not\in Y_n$;
if $\pi(Y_n) \neq C$, take $y \in T \setminus \{\pi(x_1),\ldots, \pi(x_{n-1})\}$ satisfying
$y \not\in \pi(Y_n)$ and take any $x_n \in X_y(L)$.
Then the set $S=\{x_1,x_2,\ldots\}$ satisfies the claim.

\vspace{0.1in}

\underline{Case 2}: $\delta_f>\delta_g$.

\noindent Write $X= \PP(\mathcal{E})$, where $\mathcal{E}$ is a locally free sheaf of rank $2$ on $C$.
We may assume that $H^0(C,\mathcal E) \neq 0$ and
$H^0(C,\mathcal E \otimes \mathcal L)=0$ for any invertible sheaf $\mathcal L$ on $C$ with $\deg \mathcal L <0$.
Let $e= -\deg \mathcal{E}$ and $C_0$ a section of $\pi$ such that
$\mathcal{O}_X(C_0)\cong \mathcal{O}_X(1)$ (see \cite[V. Proposition 2.8]{Har77}).
Let $F$ be a fiber of $\pi$.

\vspace{0.1in}

\underline{Case 2.a}: $\mathcal{E}$ is decomposable i.e. $\mathcal{E} \cong \mathcal{O_C} \oplus \mathcal{L}$
for an invertible sheaf $\mathcal{L}$ on $C$.

\noindent Then $D_0 =eF+C_0$ is a nef Cartier divisor satisfying $f^\ast D_0 \equiv \delta_f D_0$
(see \cite[Lemma 5.9 and Lemma 5.10]{MSS18a}).
We can take a nef $\R$-Cartier divisor $D \equiv D_0$ satisfying $f^*D \sim_\R \delta_f D$ by Theorem \ref{thm_eigen} (2).
Now $D$ has positive intersection number with any fiber of $\pi$.
Hence $(X,f)$ satisfies the assumption ($\dag$) of Theorem \ref{thm_unirat2} and
so the assertion follows by Theorem \ref{thm_unirat2}.

\vspace{0.1in}

\underline{Case 2.b}: $\mathcal{E}$ is indecomposable.

\noindent Then $\mathcal E$ is semistable (see \cite[10.2 (c), 10.49]{Muk02}).
Applying \cite[Theorem 2 and Proposition 2.4]{Ame03} and \cite[Lemma 6.3]{MSS18a},
there is an endomorphism $q: C \to C$ such that
the base change $\pi':X' \to C$ of $\pi:X\to C$ along $q$ is a $\PP^1$-bundle
defined by a decomposable locally free sheaf
and there is a morphism $g':C \to C$ satisfying $q \circ g'=g \circ q$.
Let $\pi': X'\to C$ and $p: X'\to X$ be the projections.
By the universality of cartesian products,
a morphism $f': X' \to X'$ satisfying
$p\circ f= f'\circ p$ and $\pi ' \circ f'= g'\circ \pi '$ is induced.
By Lemma \ref{lem_cover}, this case is reduced to Case 2.a.
%Now the numerical class of $F$ is lying on one of the extremal rays of the nef cone of $X$
%(note that the dimension of $N^1(X)$ is two)
%and satisfies $f^* F \equiv \delta_g F$.
%Let $D_0$ be a divisor whose numerical class is
%lying on the other extremal ray of the nef cone of $X$.
%Then it satisfies $f^\ast D_0 \equiv \delta_f D_0$ since $\delta_g<\delta_f$ and
%$\rho(f^*)=\delta_f$.
%For a divisor $D$ on $X$, $D$ is nef if and only if $p^\ast D$ is nef.
%Hence $p^\ast$  sends bijectively the extremal rays of the nef cone of $X$
%to those of $X'$.
%Let $F'$ be a fiber of $\pi'$ and
%$D_0'$ a divisor on $X'$ such that $\{F',D_0'\}$ generate the nef cone of $X'$.
%Then $p^\ast F \equiv \deg q \cdot F'$ and
%$p^\ast D_0 \equiv t D_0'$ for some $t>0$.
%As we saw in Case 2.a, $F'$ has positive intersection with $D_0'$.
%Hence we have the following equalities:
%\begin{align*}
%&\hphantom{=} \deg p \cdot (F \cdot D_0) = (p^\ast F \cdot p^\ast D_0)
%= \deg q \cdot (F' \cdot p^\ast D_0)\\
%&= t\cdot \deg q \cdot (F' \cdot D_0') >0.
%\end{align*}
%We can take a nef $\R$-Cartier divisor $D \equiv D_0$ satisfying $f^*D \sim_\R \delta_f D$ by Theorem \ref{thm_eigen} (2).
%Hence $(X,f)$ satisfies the assumption ($\dag$) of Theorem \ref{thm_unirat2} and so
%the assertion follows by Theorem \ref{thm_unirat2}.
\end{proof}

\appendix
\def\thesection{\Alph{section}}
\section{On a theorem of Kawaguchi--Silverman}\label{sec_app}
Notation and conventions in this appendix are as above.
Other notation follows \cite{KS14}.
In this appendix, we show the following.

\begin{thm}[cf.~{\cite[Theorem 3]{KS14}}]\label{thm_richness}
	Let $f: \A^2_K \to \A^2_K$ be a dominant rational map with $\delta_f>1$.
	Assume that either of the following is true:
		\begin{parts}
			\Part{(a)} $f^m$ is algebraically stable for some $m>0$.
			\Part{(b)} $f$ is a quadratic map.
		\end{parts}
Then there is a finite extension field $L \supset K$ such that
$(\A^2_K,f)$ has densely many $L$-rational points with maximal arithmetic degree.
\end{thm}

In the above setting,
Kawaguchi and Silverman showed that 
$(\A^2_K,f)$ has densely many $\overline K$-rational points with maximal arithmetic degree (\cite[Theorem 3]{KS14}).
So Theorem \ref{thm_richness} is a stronger version of \cite[Theorem 3]{KS14}.

\subsection{Preparatory results}
Here we collect some preparatory notions and results.

\begin{defn}
Let $f:  \A^N_K \to \A^N_K$ a dominant morphism with $\delta_f>1$.
	For $P=(x_1, x_2, \ldots , x_N) \in \A^N(K)$ and $v \in M_K$,
	we define $\lambda_v: \A^N(K) \to \R$ and $\underline{\lambda}_{f,v}: \A^N(K)\to \R \cup \{\infty\}$ by
	\begin{align*}
		\lambda_v(P) &= \log \max\{ \| x_1 \|_v, \| x_2 \|_v, \ldots , \| x_N \|_v, 1 \},\\
		\underline{\lambda}_{f,v}(P) &= \liminf_{n\to \infty} \delta_f^{-n} \lambda_v(f^n(P)).
	\end{align*}
\end{defn}

\begin{rem}\label{rem_eq}
	If $\underline{\lambda}_{f, v_0}(P)>0$ for some $v_0 \in M_K$,
	then
	\[
	h(f^n(P))= \sum_{v\in M_K} \lambda_{v} (f^n(P)) \geq \lambda_{v_0} (f^n(P)) \succeq \delta_f^n.\]
	Hence we have
	\[
	\delta_f \geq \overline{\alpha}_f(P) \geq \underline{\alpha}_f(P) \geq \delta_f,
	\]
	where the first inequality is a result of \cite{Mat16}.
	Therefore $\alpha_f(P)=\delta_f$ in this case.
\end{rem}

%\begin{lem}\label{lem_relation}
%	Let $\phi :\A^N_K \to \A^N_K$ be a morphism of degree at most $d$.
%	Then the inequality
%$$
%		\lambda_v(\phi(P)) \leq d \lambda_v (P)
%$$
%	holds for all but finitely many $v\in M_K$ and all $P \in \A^N(K_v)$
%\end{lem}
%
%\begin{proof}
%Write $\phi=(\phi_1, \ldots ,\phi_N)$ with $\phi_i \in K[X_1,\ldots,X_N]\ (1\leq i \leq N)$
%and $P=(x_1,x_2,\ldots,x_N)$.
%Let $v$ be a non-archimedean place of $K$.
%If all coefficients of $\phi_i$ are unit at $v$, we have the inequality
%	\[
%	\|\phi_i(P)\|_v \leq (\max\{ \|x_1\|_v, \ldots ,\|x_N\|_v, 1\})^d
%	\]
%Taking logarithm and maximum, we are done.
%\end{proof}
%
%\begin{lem}
%	For a linear transformation $\phi :\A^N_K \to \A^N_K$,
%	we have the equality
%	\[
%	\lambda_v \circ\phi = \lambda_v
%	\]
%	on $\A^N(K_v)$ for all but finitely many $v\in M_K$.
%\end{lem}
%
%\begin{proof}
%	Apply Lemma \ref{lem_relation} to $\phi$ and $\phi^{-1}$.
%\end{proof}

\begin{lem}\label{lem_num2}
	Let $d\in \Z_{>1}$ and $l_1,\ldots,l_k \in \Z_{>0}$.
	Then there exist infinitely many $l \in \Z_{>0}$ such that
	\[
	\frac{l_i}{l} \not\in  d^\Z =\{  d^n \mid n \in \Z \}.
	\]
\end{lem}

Its proof is similar to Lemma \ref{lem_num}.

%%
%%
%%
%%
%\subsection{Zariski density of the set of the points of given absolute values}
%%
%%
%%
%%
%In this subsection,
%we prove that the set of the points of given absolute values are Zariski dense.
%This might be well-known fact, but we give a proof for the reader's convenience.

\begin{prop}\label{prop_dense}
	Let $v \in M_K$ be a non-archimedean place.
	Fix an element $\pi_v\in K$ such that $\| \pi_v \|_v = \max_{y \in \mathfrak{m}_{v}} \| y\|_v$, that is, $\pi_v$ is a uniformizer of $R_v$.
	For $\mathbf{a}= (a_1,a_2, \ldots , a_N) \in \Z^N$,
	let
		\[
		U_{\mathbf{a}}= \left\{ (x_1, x_2, \ldots ,x_N)\in \A^N(K) \mid
		\| x_i\|_v = \| \pi_v\|_v^{a_i} (1\leq i \leq N)\right\}.
		\]
		Then for any  $\mathbf{a}\in \Z^N$, $U_{\mathbf{a}}$ is Zariski dense in $\PP^N_{\overline K}$.
\end{prop}

This proposition follows immediately from the following lemma.

\begin{lem}
  Let $S_1, \ldots, S_N$ be infinite subsets of $\overline{K}$.
  Then $S=\prod_{i=1}^{N}S_i$ is Zariski dense in $\A^N_{\overline K}$.
\end{lem}

\begin{proof}
  It is enough to show that for any nonzero polynomial
  $f\in \overline{K}[X_1, X_2, \ldots ,X_N]$,
  there is a point $(a_1, a_2, \ldots a_N) \in S$ where $f$ does not vanish.
  We prove the assertion by induction.
  If $N=1$, the assertion is trivial,
  since the number of zeros of $f$ is finite.

  In general, write
  \[
    f(X_1, \ldots X_N) = \sum_{i=0}^{d} f_i(X_2, X_3, \dots X_N) X_1^i
  \]
  with $f_i \in \overline{K}[X_2, X_3, \ldots X_N]$.
  By the induction hypothesis,
  we can find $(a_2, a_3, \ldots a_N) \in \prod_{i=2}^N S_i$
  at which $f_d$ does not vanish.
  So the polynomial $f(X_1, a_2, a_3, \ldots , a_N)$ is non-zero.
  Hence by the induction hypothesis for $N=1$,
  we can find $a_1 \in S_1$ such that $f(a_1, a_2, \ldots , a_N)\neq 0.$
\end{proof}

\subsection{Algebraically stable case}
For a self-morphism $f: \A^N_K \to \A^N_K$, we write $f$ also for the extension of $f$ to the rational map $\PP^N_K \dashrightarrow \PP^N_K$ by abuse of notation.
The following proposition is a generalization of \cite[Lemma 21]{KS14},
which is used several times in the proof of Theorem \ref{thm_richness} (a).

\begin{prop}\label{prop_prep}
Let $f : \A^N_K \to \A ^N_K$ be a dominant morphism
with $\delta_f>1$.
Assume that
$f^m$ is algebraically stable and has a fixed point $Q_0 \in \PP^N(K) \setminus \A^N(K)$
for some $m\geq 1$.
Then $(\A^N_K,f)$ has densely many $K$-rational points with maximal arithmetic degree.
\end{prop}

%\begin{rem}
%By the property \eqref{con_3}, we have
%$\underline{\lambda}_{f,v}=\lambda_v$
%on $U\cap \A^N(K_v)$.
%\end{rem}

\begin{proof}
We devide the proof into two steps.

\begin{step}
First we take a suitable $v \in M_K$
and a $v$-adic neighborhood $U \subset \PP^N(K_v)$ of $Q_0$ satisfying the following properties:
%	\begin{enumerate}
%		\item[(i)] $f^m(U \cap \A^N(K_v)) \subset U\cap \A^N(K_v)$.\label{con_1}
%		\item[(ii)] $\lambda_v(P)= \log \| x_1 \|_v >0$
%		for $P=(x_1, x_2, \ldots ,x_N) \in U\cap \A^N(K_v)$.\label{con_2}
%		\item[(iii)] $\lambda_v(f^n(P)) = d^n\lambda_v(P)$ for $P\in U \cap \A^N(K_v)$ and $n\geq 0$.\label{con_3}
%	\end{enumerate}

	\begin{enumerate}
		\item[(i)] $f^m(U \cap \A^N(K_v)) \subset U\cap \A^N(K_v)$.
		\item[(ii)] $\alpha_f(P)=\delta_f$ for any $U \cap \A^N(K)$.
	\end{enumerate}

Set $g=f^m$.
Let $[X_1:X_2:\cdots X_N:W]$ be homogeneous coordinates of $\PP^N_K$
and identify $\A^N_K$ with the locus $\{ W\neq 0\}$.
Conjugating $g$ by a linear transformation, we may assume $Q_0=[1:0:\cdots :0]$.
Since $Q_0$ is fixed by $g$, we can write $g$ as
	\[
	[aX_1^l+ G_1(X,W): G_2(X, W) :\cdots :G_N(X, W): W^l],
	\]
where $l=\deg g$, $a \neq 0$, and $G_i$ $(1\leq i\leq N )$ are elements of the ideal $(X_2,\ldots ,X_N, W)$ of $K[X_1,\ldots ,X_N, W]$.
Take $v\in M_K$ such that the all coefficients of $g$ are in $R_v^\ast$, and
consider the locus
	\[
	U %:= \left\{ [x_1:\cdots :x_N:w] \in \PP^N(K_v)
	%	\mid x_1\in R_v^\ast, x_2,\ldots x_N, w \in \mathfrak{m}_v\right\}\\
	= \left\{ [x_1 :\cdots :x_N: w] \in \PP^N(K_v)
		\setmid
			\begin{minipage}{8em}
				\begin{align*}
					&x_1 \neq 0,\ \left\| w/x_1\right\|_v < 1,\\
					&\left\| x_i/x_1\right\|_v <1\ (2\leq i \leq N)
				\end{align*}
			\end{minipage}
		\right\}.
	\]

Then it follows by an easy calculation that $g(U) \subset U$
(cf.~\cite[Proof of Lemma 21]{KS14}).
One can see that
	\[
	U\cap \A^N(K_v)= \left\{ (x_1,x_2, \ldots x_N) \in \A^N(K_v) \setmid
	\begin{minipage}{8em}
		\begin{align*}
			\|x_1\|_v &> 1,\\
			\|x_1\|_v &> \|x_i\|_v (2\leq i \leq N)
		\end{align*}
	\end{minipage}
	\right\},
	\]
so $\lambda_v(P)= \log \| x_1 \|_v >0$ for any $P \in U\cap \A^N(K_v)$.

Now we have
	\[
	\lambda_v(g(P))= \delta_g \lambda_v(P) \text{ for }P\in U \cap \A^N(K_v)
	\]
(cf.~\cite[Proof of Lemma 21]{KS14}).
So
$\underline{\lambda}_{g,v}(P) =\lambda_v(P)>0$ for  any $P \in U\cap \A^N(K_v)$.
Then it follows that $\alpha_g(P)=\delta_g$ for any $P \in U \cap \A^N(K)$
(cf.~Remark \ref{rem_eq}).
Since $\alpha_g(P)=\alpha_f(P)^m$ and $\delta_g=\delta_f^m$,
we have $\alpha_f(P)=\delta_f$ for any $P \in U \cap \A^N(K)$.
\end{step}

\begin{step}
According to Lemma \ref{lem_key} and Remark \ref{rem_eq},
it is sufficient to prove the following claim, where we set 
\begin{itemize}
\item $x \in \A^N(K)$ satisfies (P) if $x \in U$.
\item (Q) is unconditional.
\end{itemize}

\begin{claim}
Given finitely many points $P_1,\ldots,P_k \in \A^N(K)$
and a proper closed
subset $Y \subset \A^N_{\overline K}$,
there exists a point $P \in U \cap \A^N(K)$ such that
$P \not\in Y$ and
$O_f(P) \cap O_f(P_i)=\varnothing$ for $1 \leq i \leq k$.
\end{claim}

\noindent Using Lemma \ref{lem_it}, we may assume $m=1$.
Replacing $P_i$ with the first point of $O_f(P_i)$ that is included in $U$ or
otherwise excluding $P_i$ for each $i$,
we may assume $P_i \in U$ for every $i$.
Fix a uniformizer $\pi_v \in K$ of $R_v$.
Let $l_i = \lambda_v(P_i)/ (-\log \| \pi_v\|_v) \in \Z_{>0} \ (1\leq i \leq k)$.
By Lemma \ref{lem_num2}, there exists $l>1$ such that $l_i/l \not \in d^\Z$.
Consider the set
	\[
	U^{(l)}:= \left\{ (x_1, x_2, \ldots , x_N) \in \A^N(K) \setmid
	\begin{minipage}{8em}
		\begin{align*}
			&\| x_1 \|_v = \| \pi_v\|_v^{-l} \text{ and}\\
			&\| \pi_v\|^{-l}_v > \| x_i \|_v\geq 1 \ (2\leq i \leq N)
		\end{align*}
	\end{minipage}
	\right\}.
	\]
	Note that $U^{(l)} \subset U \cap \A^N(K)$.
The set $U^{(l)}$ is Zariski dense by Proposition \ref{prop_dense},
so there is an element $P \in U^{(l)} \setminus Y$.
Then we have
	\[
	\frac{\lambda_v{(P_i)}}{\lambda_v(P)}
	= \frac{\lambda_v(P_i)}{-\log \| \pi_v\|_v} \cdot \frac{-\log \| \pi_v\|_v}{\lambda_v(P)}
	= \frac{l_i}{l} \not \in \pm d^\Z,
	\]
so $O_f(P) \cap O_f(P_i)=\varnothing$ for $1 \leq i \leq k$.
\end{step}
\end{proof}

\begin{proof}[Proof of Theorem \ref{thm_richness} (a)]
We prove Theorem \ref{thm_richness} (a) along \cite[Proof of Theorem 18 (a)]{KS14}.
Let $d=\deg f$.
We write $f$ in homogeneous coordinates as
	\begin{align*}
		f([X:Y:Z]) =[ &F(X,Y) + Z F_1(X,Y,Z):\\
 			&\hphantom{=} G(X,Y)+ ZG_1(X,Y,Z): Z^d ],
	\end{align*}
where at least one of $F$ and $G$ is nonzero, since $f$ has degree $d$.
Changing coordinates, we may assume $F\neq 0$.

Let $H= \gcd (F,G) \in K[X,Y]$, and write $F=HF_0$ and $G=HG_0$.
Set $V=\{ Z=0 \} =\PP^2_K \setminus \A^2_K$.
Since we have
	\[
	\deg F_0= \deg F- \deg H= d- \deg H= \deg G- \deg H = \deg G_0,
	\]
we have a well-defined map
	\[
	\phi=[F_0,G_0] \colon V \longrightarrow V.
	\]

\vspace{0.1in}

\underline{Case 1}: $\deg F_0 \geq 2.$

	\noindent Since the degre of $\phi$ is at least $2$,
	it has infinitely many periodic points in $V(\overline{K})$.
	Hence there are infinitely many periodic points of $f$ in $V(\overline K)$.
	By taking a finite extension $L$ of $K$, we can get a periodic point $Q_0 \in V(L)$.
	Hence we can apply Proposition \ref{prop_prep}.

	\vspace{0.1in}

\underline{Case 2}: $\deg F_0 =0.$

	\noindent By conjugating $f$ by an automorphism, we may assume that
	$f$ is of the form
		\[
		f([X:Y:Z])=[H(X,Y)+ZF_1(X,Y,Z): ZG_1(X,Y,Z):Z^d].
		\]
	(cf.~the proof of \cite[Theorem 18 (a)]{KS14} for details.)

	\vspace{0.1in}

		\underline{Case 2.a}: $H(1,0) =0.$

		\noindent In this case, any positive power of $f$ is not algebraically stable (cf.~ the proof of
		 \cite[Theorem 18 (a)]{KS14}),
		 which cannot happen.

		 \vspace{0.1in}

		\underline{Case 2.b}: $H(1,0) \neq 0.$

		 \noindent The map $f$ is defined at $[1:0:0]$,
		 and $[1:0:0]$ is a fixed point of $f$,
		 so we can apply Proposition \ref{prop_prep}.

\vspace{0.1in}

\underline{Case 3}: $\deg F_0 =1$.

	\noindent In this case, after a change of coordinates of $\PP^2$ mapping the line at infinitiy to itself,
	the map $\phi$ may be put into one of the following two forms:
	\[
	\phi(X,Y)=[aX:Y] \text{ or } \phi(X,Y) = [X+bY: Y] \text{ with }a,b \in \overline{K}.
	\]

	Let $v$ be a non-alchimedean place of $K$ such that
	the nonzero coordinates of $H, F_0,F_1,G_0,$ and $G_1$ have $v$-adic absolute value $1$.
	Then consider the $v$-adic open set
	\[
	U= \left\{
		P=[x:y:z]\in \PP^2(K_v) \mid
		|x|_v > |y|_v > |z|_v\text{ and }
		|y|_v^d > |x|_v^{d-1} |z|_v
		\right\}.
	\]
	For a point $P=[\alpha: \beta:1] \in U \cap \A^2(K)$,
	we can write $|\beta|_v = R$, $|\alpha|_v =RS$
	with $R>1$, $S>1$, and  $R>S^{d-1}$
	by the condition of $U$.
	Let	$f(P)=[\alpha ' : \beta ': 1]$.

	\vspace{0.1in}

		\underline{Case 3.a}: $\phi =[aX:Y]$.

		\noindent In the proof of \cite[Theorem 3]{KS14},
		it is showed that $f(U) \subset U$,
		$|\alpha'|_v = R^d S^{k+1}$, and
		$|\beta'|_v = R^d S^k$,
		where $k$ is the smallest integer such that
		the coefficient of $X^kY^{d-1-k}$ of $H$ is non-zero.
		The last two equalities imply
		$\alpha_f(P)= d$
		for all $P\in U \cap \A^2(K) $.

According to Lemma \ref{lem_key}, it is sufficient to show the following claim,
where we set
\begin{itemize}
\item $x \in \A^2(K)$ satisfies (P) if $x \in U$.
\item (Q) is same as (P).
\end{itemize}

\begin{claim}
Gven finitely many points $P_1,\ldots,P_k \in  U \cap \A^2(K)$
and a proper closed
subset $Y \subset \A^2_{\overline K}$,
there exists a point $P \in U \cap \A^2(K)$ such that
$P \not\in Y$, and
$O_f(P) \cap O_f(P_i)=\varnothing$ for $1 \leq i \leq k$.
\end{claim}

		\noindent
		Write $P_i=(\alpha_i, \beta_i)$,
		$|\alpha_i|_v = R_i S_i$ and
		$|\beta_i|_v = R_i$ with
		$R_i > 1, S_i>1, \text{ and } R_i>S_i^{d-1}$.
		Now let $S>1$ be a value of $|\cdot|_v$
		which is different from $S_i\ (1 \leq\ i \leq k)$.
		Fix a value $R>1$ of $|\cdot|_v$ satisfying
		$R > S^{d-1}$.
		Now pick a point $P=(\alpha,\beta) \in U \cap \A^2(K)$
		such that
		$|\alpha|_v = RS$,
		$|\beta|_v = S$, and
		$P \not\in Y$.
		Note that such a point exists by Proposition \ref{prop_dense}.
		Now note that the value $|\alpha'/\beta'|_v$ is constant
		for any point $(\alpha',\beta')$ on the forward $f$-orbit of a point of $U$.
		So we have $O_f(P) \cap O_f(P_i)=\varnothing$ for $1 \leq i \leq k$.

		\vspace{0.1in}

		\underline{Case 3.b}: $\phi =[X+bY:Y]$.

		\noindent In this case, the equality
		$|\beta'|_v = |\beta|_v^d$
		is proved in the proof of \cite[Theorem 3 (a)]{KS14}.
		This allows us to show the assertion
		in a similar way as the proof of Proposition \ref{prop_prep}.
\end{proof}

\subsection{Quadratic case}
\begin{proof}[Proof of Theorem \ref{thm_richness} $(b)$]
	By the result of \cite{Gue04}, algebraically non-stable quadratic map
	$f:\A^2_K \to \A^2_K$ is conjugated by a $\overline{K}$-linear automorphism $\phi$ of $\A^2_{\overline K}$ to one of the following maps:
	\begin{parts}
		\Part{Case 1.1.} $f(x,y)=(y+c_1,xy+c_2)$ with $c_1,c_2 \in \overline{K}$,
		\Part{Case 3.1.} $f(x,y)=(y, x^2+ ax+c)$ with $a,c \in \overline{K}$, or
		\Part{Case 3.2.} $f(x,y)=(ay+c_1,x(x-y)+c_2)$ with $a\in \overline{K}^{\ast}, c_1,c_2 \in \overline{K}.$
	\end{parts}
	We prove the assertion by referring to the calculations in the proof of \cite[Theorem 18 $(b)$]{KS14} and case by case analyses.
	Replacing the base field $K$ by a finite extension,
	we may assume that $\phi$ and the resulting morphism $f$ are defined over $K$.
	Take $v \in M_K$ such that all coefficients of $f$ are in $R_v^*$.

	\vspace{0.1in}

		\underline{Case 1.1}: $f(x,y)=(y+c_1,xy+c_2)$ with $c_1,c_2 \in K$.

			\noindent For $P=(x,y)\in \A^2(K)$ with $|x|_v= |y|_v> 1$.
			In the case $1.1$ of the proof of \cite[Theorem 18 $(b)$]{KS14},
			it is proved that
			\[
			 \underline{\lambda}_{f,v}(P)=
			 \frac{\delta_f^2 \log |x|_v}{\sqrt{5}} >0.
			\]
			Hence the same argument as in Step $2$ of the proof of Proposition \ref{prop_prep} works.

			\vspace{0.1in}

		\underline{Case 3.1}: $f(x,y)=(y, x^2+ ax+c)$ with $a,c \in K$.

			\noindent In this case, $f^2(x,y)=(x^2+ax+c,y^2+ay+c)$ is algebraically stable. So it is reduced to Theorem \ref{thm_richness} (a).

			\vspace{0.1in}

		\underline{Case 3.2}: $f(x,y)=(ay+c_1,x(x-y)+c_2)$ with $a\in K^{\ast}, c_1,c_2 \in K.$

			\noindent For $P=(x,y) \in \A^2(K)$ with $1< |x|_v < |y|_v$.
			In the case $3.2$ of the proof of \cite[Theorem 18 $(b)$]{KS14},
			it is proved that
			\[
			 \underline{\lambda}_{f,v}(P)=
			 \frac{\delta_f^2 \log |y|_v}{\sqrt{5}} >0.
			\]
			Hence the same argument as the Step $2$ in the proof of Proposition \ref{prop_prep} works.

\end{proof}

\end{document}